\documentclass[a4paper,twoside]{amsart}
\usepackage[utf8]{inputenc}
\usepackage[hidelinks]{hyperref}
\usepackage{graphicx}
\usepackage{enumitem}
\usepackage{mathrsfs}
\usepackage{etoolbox}
\usepackage{amssymb}
\usepackage{xcolor}

\def\Aut{\mathop{\mathrm{Aut}}\nolimits}

\def\str#1{\mathbf {#1}}

\theoremstyle{plain}
\newtheorem{theorem}{Theorem}[section]

\newtheorem{prop}[theorem]{Proposition}
\newtheorem{observation}[theorem]{Observation}
\newtheorem{lemma}[theorem]{Lemma}

\theoremstyle{definition}
\newtheorem{definition}[theorem]{Definition}
\newtheorem{example}[theorem]{Example}
\newtheorem*{remark*}{Remark}
\newtheorem*{claim*}{Claim}
\theoremstyle{remark}

\def\str#1{\mathbf {#1}}
\def\Fraisse{Fra\"{\i}ss\' e}
\def\Aut{\mathop{\mathrm{Aut}}\nolimits}
\def\Emb{\mathop{\mathrm{Emb}}\nolimits}

\def\Monoid{{\mathfrak M}}
\def\EMonoid{\Monoid=(M,\oplus,\mleq,0)}
\def\mleq{\preceq}

\def\mgeq{\succeq}

\begin{document}
\title{Ramsey expansions of $S$-metric spaces}
\authors{
\author[J. Hubi\v cka]{Jan Hubi\v cka}
\address{Charles University, Faculty of Mathematics and Physics\\Department of Applied Mathematics (KAM)\\
Prague, Czech Republic}
\email{hubicka@kam.mff.cuni.cz}
\author[M. Kone\v cn\'y]{Mat\v ej Kone\v cn\'y}
\address{Charles University, Faculty of Mathematics and Physics\\Department of Applied Mathematics (KAM)\\
Prague, Czech Republic}
\email{matej@kam.mff.cuni.cz}
\author[J. Ne\v set\v ril]{Jaroslav Ne\v set\v ril}
\address{Charles University, Faculty of Mathematics and Physics\\
Computer Science Institute of Charles University (IUUK)\\
Prague, Czech Republic}
\email{nesetril@iuuk.mff.cuni.cz}
\author[N. Sauer]{Norbert Sauer}
\address{Mathematics \& Statistics Department\\
University of Calgary\\
Calgary, T2N 1N4, Alberta, Canada}
\email{nsauer@math.ucalgary.ca}

\thanks{J.H. and J.N\ were supported by project 25-15571S of the  Czech Science Foundation (GA\v CR). This paper is part of a project that has received funding from the European Research Council (ERC) under the European Union's Horizon 2020 research and innovation programme (grant agreement No 810115 -- {\sc Dynasnet}).
}

\dedicatory{In memory of our good friend Robert Woodrow}

\begin{abstract}
Let $S\subseteq \mathbb R^+$ be a subset of positive reals. We say that a metric space is an \emph{$S$-metric space} if the range of the metric is a subset of $S\cup \{0\}$. In 2007, Delhomm{\'e}, Laflamme, Pouzet, and Sauer characterized those sets $S$ for which the class of all finite $S$-metric spaces has the amalgamation property. Given an amalgamation class, a natural question is what its Ramsey expansion is. Ramsey expansions are known for classes of $S$-metric spaces for which $S\cup\{0\}$ carries a natural distance monoid structure, for example for $S=\mathbb Q^+$ or when $S$ is topologically closed. In this paper, we extend this result to all amalgamation classes of finite $S$-metric spaces, thereby completing the Ramsey-theoretic classification of amalgamation classes of finite $S$-metric spaces.
\end{abstract}
\maketitle

\section{Introduction}
A structure $\str M$ is \emph{homogeneous} if every isomorphism between finite substructures of $\str M$ extends to an automorphism of $\str M$. A class $\mathcal C$ of finite structures has the \emph{amalgamation property} if for every $\str A, \str B_1,\str B_2\in \mathcal C$ such that $\str A$ is a substructure of both $\str B_1$ and $\str B_2$, there is $\str C\in \mathcal C$ such that both $\str B_1$ and $\str B_2$ are substructures of $\str C$, sharing $\str A$ as a common substructure. The interplay between (infinite) homogeneous structures and classes with the amalgamation property is the subject of the famous \Fraisse{} theory~\cite{Fraisse1986}. In particular, the \emph{age} (class of all finite substructures) of every homogeneous structure has the amalgamation property, and, under certain conditions, classes with the amalgamation property yield infinite homogeneous structures.

It turns out that many group-theoretic or model-theoretic properties of homogeneous structures are reflected in combinatorial properties of their ages. One of the most prominent is the Ramsey property:

\begin{definition}\label{defn:ramsey}
A class $\mathcal C$ of finite structures is \emph{Ramsey} (has the \emph{Ramsey property}) if for every $\str{A}\in \mathcal C$ and $\str{B}\in\mathcal C$, there exists $\str{C}\in\mathcal C$ such that the following holds: For every colouring $c\colon \Emb(\str A, \str C)\to \{0,1\}$ of all embeddings of $\str A$ into $\str C$ by 2 colours, there exists an embedding $f\colon \str B\to \str C$ such that the colouring $c$ is constant on $f\circ \Emb(\str A, \str B) = \{f\circ g : g\in \Emb(\str A, \str B)\}$.
\end{definition}
It is usual to denote the last part of Definition~\ref{defn:ramsey} as $\str{C} \longrightarrow (\str{B})^{\str{A}}_2$.

Structural Ramsey theory emerged in the 1970s from work of Leeb~\cite{Leeb} and Ne\v set\v ril and R\"odl~\cite{Nevsetvril1977}, generalizing Ramsey's theorem~\cite{Ramsey1930} from finite sets to finite structures. The area was revitalized in 2005 by the influential paper of Kechris, Pestov and Todor\v cevi\'c~\cite{Kechris2005}, connecting Ramsey classes to topological dynamics and the study of universal minimal flows. Since then, many new Ramsey classes have been found and general machinery for proving the Ramsey property has been developed, see e.g.~\cite{Hubicka2016} or the recent survey~\cite{hubicka2025twenty}.

\medskip
 
Among classes with the amalgamation property that have been studied from the perspective of the Ramsey property, metric and metric-like spaces occupy a prominent place. Ne\v set\v ril~\cite{Nevsetvril2007} proved that the class of all finite linearly ordered metric spaces is Ramsey. The fact that one needs to add some variant of a linear order is a well-understood phenomenon in Ramsey theory, see e.g.~\cite{Kechris2005,Dellamonica2012}. This result served as a foundation for further research on Ramsey properties of metric and metric-like spaces, see e.g.~\cite{The2010,Dellamonica2012,Hubicka2016,Hubicka2017sauerconnant,Sam,Aranda2017c,Aranda2017,Konecny2018b}.

Delhomm{\'e}, Laflamme, Pouzet, and Sauer in 2007~\cite{Delhomme2007} characterized precisely for which subsets $S\subseteq \mathbb R^+$ the class of all finite metric spaces whose range of distances is a subset of $S\cup \{0\}$ (we will call such spaces \emph{$S$-metric spaces}) has the amalgamation property. It turns out that this can be neatly expressed by a simple condition on $S$ (which in fact says that amalgamations of triangles exist):
\begin{definition}[\cite{Delhomme2007}]
A subset $S\subseteq \mathbb R^+$ satisfies the \emph{4-values condition} if for every $a,b,c,d\in S$, whenever there is an element $x\in S$ such that the triangles with distances $(a,b,x)$ and $(c,d,x)$ satisfy the triangle inequality, there is $y\in S$ such that the triangles with distances $(a,c,y)$ and $(b,d,y)$ satisfy the triangle inequality.
\end{definition}
The characterization of Delhomm{\'e}, Laflamme, Pouzet, and Sauer~\cite{Delhomme2007} can now be stated as follows.
\begin{theorem}[\cite{Delhomme2007}]\label{thm:4values}
Let $S$ be a subset of positive real numbers and let $\mathcal M_S$ be the class of all finite $S$-metric spaces. Then $\mathcal M_S$ has the amalgamation property if and only if $S$ satisfies the 4-values condition.
\end{theorem}

Nguyen Van Th\'e~\cite{NVT2009} proved that the Ramsey property holds for various convexly ordered $S$-metric spaces. His results were later significantly extended by Hubi\v cka and Ne\v set\v ril~\cite{Hubicka2016}, who covered all cases where $S\cup\{0\}$ carries a natural distance monoid structure (for example when $S=\mathbb Q^+$ or when $S$ is topologically closed).

In this paper, we conclude the study of Ramsey expansions for $S$-metric spaces by proving the following theorem.
\begin{theorem}\label{thm:main}
Let $S$ be a subset of positive real numbers satisfying the 4-values condition, and let $\overrightarrow{\mathcal M}_S$ be the class of all finite $S$-metric spaces equipped with a linear order convex with respect to block equivalences. Then $\overrightarrow{\mathcal M}_S$ has the Ramsey property and the expansion property with respect to $\mathcal M_S$.
\end{theorem}

To prove Theorem~\ref{thm:main} we use the Ramsey theorem for metric spaces over distance monoids proved by Hubi\v cka, Kone\v cn\'y and Ne\v set\v ril~\cite{Hubicka2017sauerconnant}. Given a set $S$ satisfying the 4-values condition, we construct a distance monoid containing a copy of $S$ and translate $S$-metric spaces to metric spaces over this monoid. We then use the Ramsey theorem and show how to ``adjust'' the distances in a Ramsey witness to get an $S$-metric space.

\medskip

\paragraph{\textbf{AI usage disclaimer.}} The original version of this paper dates from 2020. ChatGPT (GPT-5.6 Sol) was used in the present revision for proof checking and editing; the authors take full responsibility for the final text.

\section{Preliminaries}
The structures considered in this paper are metric or metric-like spaces (that is, they have a vertex set and a distance function), possibly equipped with an order. Equivalently, these structures can be understood as relational structures in a binary language, with one binary symmetric relation for each non-zero distance (and, possibly, an order). The embeddings are (monotone) distance-preserving injections and correspond to the classical embeddings of relational structures.

\subsection{Ramsey expansions}\label{sec:expansions}
As we have already mentioned, classes with the amalgamation property are often not Ramsey by themselves. One then adds some extra structure, often a linear order, to obtain a Ramsey class. To rule out vacuous expansions, Nguyen Van Th\'e~\cite{NVT2013} introduced the expansion property. Since in this paper we only add a linear order, we state the definition directly in this setting.

\begin{definition}[Expansion property~\cite{NVT2013}]\label{defn:expproperty}
Let $\mathcal C$ be a class of finite structures and let $\mathcal C^+$ be a class of structures obtained from members of $\mathcal C$ by adding a linear order. We say that $\mathcal C^+$ has the \emph{expansion property} (with respect to $\mathcal C$) if for every $\str B\in \mathcal C$ there is $\str C\in \mathcal C$ such that every expansion $\str B^+\in\mathcal C^+$ of $\str B$ embeds into every expansion $\str C^+\in\mathcal C^+$ of $\str C$.
\end{definition}

\medskip

In our case, we add certain linear orders which behave well with respect to the natural equivalence relations given by the distances. For the rest of this subsection, assume that $S$ satisfies the 4-values condition.

We define a binary relation $\sim$ on $S$ by saying that $a\sim b$ if and only if there is a sequence $x_1,x_2,\ldots,x_m$ such that $x_i\in S$ for every $i$, $x_1 = \min(a,b)$, $x_m\geq\max(a,b)$ and $x_i-x_{i-1} \leq \min(a,b)$ for every $2\leq i\leq m$.  It is straightforward to verify using the 4-values condition that $\sim$ is an equivalence relation and that if $a\sim b$ and $c\in S$ is such that $a\leq c\leq b$, then $a\sim c$ (i.e. the equivalence classes of $\sim$ form intervals in the order on $S$). We will call the equivalence classes of $\sim$ \emph{blocks} of $S$. For finite $S$, the same block decomposition appears in~\cite{Sauer2012}.

\begin{definition}\label{defn:s_convex_order}
Let $\str A = (A, d_\str A)$ be an $S$-metric space and let $B$ be a block of $S$. We define a \emph{block equivalence} $\sim_B$ on vertices of $\str A$ by putting $u\sim_B v$ if and only if there is $b\in B$ such that $d_\str A(u,v)\leq b$. We say that a linear order $\leq_\str A$ of $A$ is \emph{convex with respect to block equivalences} if for every block $B$ of $S$ it holds that every equivalence class of $\sim_B$ forms an interval in $\leq_\str A$.
\end{definition}
The previous definition extends the definition of convex ordering on $S$-metric spaces from~\cite{Hubicka2016}.

\begin{example}
The set $S=\{1,3,7\}$ satisfies the 4-values condition and has three blocks $\{1\}$, $\{3\}$, and $\{7\}$. The set $S'=\{1,3,5,7\}$ also satisfies the 4-values condition and has two blocks, $\{1\}$ and $\{3,5,7\}$.
\end{example}

\subsection{Distance monoids}
In this subsection, we recall the definitions and the Ramsey theorem from~\cite{Hubicka2017sauerconnant}.
When $S$ is finite, the connection between an associative operation of this form and the amalgamation property for $S$-metric spaces was already observed by Sauer~\cite{sauer2013b}.

\begin{definition}[\cite{Hubicka2017sauerconnant}]
A \emph{distance monoid} is a tuple $\EMonoid$ such that
\begin{enumerate}
\item $(M,\oplus,0)$ is a commutative monoid with identity $0$;
\item $(M,\mleq)$ is a linear order with least element $0$;
\item\label{def:POS:monot} if $a\mleq b$ and $c\mleq d$, then $a\oplus c\mleq b\oplus d$.
\end{enumerate}
\end{definition}
Note that the last point and the fact that $0$ is the least element imply $a\mleq a\oplus b$ for all $a,b\in M$.

\begin{definition}[\cite{Hubicka2017sauerconnant}]
Let $\EMonoid$ be a distance monoid. Given a set $A$ and a function $d\colon A^2\to M$, we call $(A,d)$ an \emph{$\Monoid$-metric space} if for every $x,y,z\in A$:
\begin{enumerate}
\item $d(x,y)=0$ if and only if $x=y$;
\item $d(x,y)=d(y,x)$; and
\item $d(x,z)\mleq d(x,y)\oplus d(y,z)$.
\end{enumerate}
\end{definition}

Let $\EMonoid$ be a distance monoid. For $a,b\in M\setminus\{0\}$ put $a\approx b$ if and only if there are positive integers $m$ and $n$ such that $m\times a\mgeq b$ and $n\times b\mgeq a$, where
$$m\times a = \underbrace{a\oplus a\oplus\cdots\oplus a}_{\text{$m$ times}}$$
and similarly for $n\times b$. The set $\{0\}$ and the equivalence classes of $\approx$ are called \emph{blocks} of $\Monoid$.

\begin{definition}\label{defn:convex_monoid}
Let $\EMonoid$ be a distance monoid and let $\str A = (A, d_\str A)$ be an $\Monoid$-metric space. Let $B$ be a block of $\Monoid$. We define a \emph{block equivalence} $\sim_B$ on vertices of $\str A$ by putting $u\sim_B v$ if and only if there is $b\in B$ such that $d_\str A(u,v)\mleq b$. We say that a linear order $\leq_\str A$ of $A$ is \emph{convex} if every equivalence class of $\sim_B$ forms an interval in $\leq_\str A$ for every block $B$ of $\Monoid$.
\end{definition}

We will use the following theorem.
\begin{theorem}[\cite{Hubicka2017sauerconnant}]\label{thm:monoids}
For every distance monoid $\EMonoid$, the class of all finite convexly ordered $\Monoid$-metric spaces is Ramsey.
\end{theorem}

\section{The result}
Throughout the section we fix a set $S\subseteq \mathbb R^+$ which satisfies the 4-values condition. We denote by $\mathfrak S$ the set of all subsets $X\subseteq S$ satisfying the following:
\begin{enumerate}
\item If $x\in X$ and $y\in S$ is such that $y\leq x$, then $y\in X$ (i.e. $X$ is downwards closed),
\item there is $r\in \mathbb R$ such that for every $x\in X$ we have $x\leq r$ (i.e. $X$ is bounded from above), and
\item $X$ is non-empty.
\end{enumerate}
For  $x\in S$ let $\hat x = \{s\in S : s\leq x\}$ and put $\widehat S=\{\hat x:x\in S\}$. A set  $X\in \mathfrak S$ is \emph{original} if $X\in\widehat S$, otherwise $X$ is \emph{virtual}.

First, we define a binary commutative operation $\oplus\colon \mathfrak S^2\to \mathfrak S$, putting
$$X\oplus Y = \{s\in S : (\exists x\in X)(\exists y\in Y)(x+y\geq s)\}.$$
Note that if $X=\hat x$ and $Y=\hat y$ are original, then $X\oplus Y = \{s\in S : s\leq x+y\}$.

It is straightforward to verify that $X\subseteq X\oplus Y$ for every $X,Y\in \mathfrak S$, and that if $Y\subseteq Z$, then $X\oplus Y \subseteq X\oplus Z$. In other words, the operation $\oplus$ is monotone with respect to set inclusion. Moreover, this inclusion is a linear order on $\mathfrak S$.

The order and monotonicity properties above show that, after adding a least neutral element $0$, the only distance monoid axiom left to verify is associativity of $\oplus$. For this, one needs the 4-values condition and an auxiliary observation.
\begin{observation}\label{obs:workswell}
For every $X,Y\in \mathfrak S$ we have $\sup(X\oplus Y) \leq \sup(X)+\sup(Y)$. In particular, for $x,y,z\in S$ we have $x+y\geq z$ if and only if $\hat x\oplus \hat y\supseteq \hat z$.
\end{observation}
\begin{proof}
By definition, $s\in X\oplus Y$ if and only if $s\in S$ and there are $x\in X$ and $y\in Y$ such that $s\leq x+y$. So in particular $s\leq \sup(X) + \sup(Y)$ for every $s\in X\oplus Y$. Hence this also holds for $\sup(X\oplus Y)$. If $x,y,z\in S$ are such that $x+y\geq z$, we know that $\sup(\hat x) = x$ and $\sup(\hat y)=y$, hence $z\in \hat x\oplus \hat y$, which means that $\hat x\oplus \hat y\supseteq \hat z$. In the other direction, if $\hat x\oplus \hat y\supseteq \hat z$, then $z\in \hat x\oplus \hat y$, so $z\leq x+y$.
\end{proof}

\begin{prop}
The operation $\oplus$ is associative; that is, for every $A,B,C\in \mathfrak S$ we have that $(A\oplus B)\oplus C = A\oplus (B\oplus C)$.
\end{prop}
\begin{proof}
We prove that $(A\oplus B)\oplus C \supseteq A\oplus (B\oplus C)$ for every $A,B,C\in \mathfrak S$. This is enough to get the conclusion, because, using commutativity, one then gets the other inequality:
$$A\oplus (B\oplus C) = (C\oplus B)\oplus A \supseteq C\oplus (B\oplus A) = (A\oplus B)\oplus C.$$

For a contradiction, suppose that there are $A,B,C\in \mathfrak S$ such that $(A\oplus B)\oplus C \subsetneq A\oplus (B\oplus C)$. This means that there is $z\in A\oplus (B\oplus C)$ such that $z\notin (A\oplus B)\oplus C$. Using the definition of $\oplus$, we get $a\in A$ and $s\in (B\oplus C)$ such that $z\leq a+s$. We may assume that $s\geq \sup(B)$ and $s\geq \sup(C)$. Indeed, if $s<\sup(B)$, then $s\in B$ and hence $z\in A\oplus B\subseteq (A\oplus B)\oplus C$, a contradiction. Similarly, if $s<\sup(C)$, then $s\in C$ and, since $A\subseteq A\oplus B$, again $z\in (A\oplus B)\oplus C$.

Clearly, $z > a$ (as otherwise $z\in A$ and by monotonicity of $\oplus$ we would have that $z\in (A\oplus B)\oplus C$.) It also holds that $z > s$: If $z\leq s$, then $z\in B\oplus C$, but by monotonicity we have that $(A\oplus B) \supseteq B$ and hence $(A\oplus B)\oplus C\supseteq B\oplus C$. Thus again $z\in (A\oplus B)\oplus C$, a contradiction. Together with $z\leq a+s$ this gives that the triangle with distances $(a,z,s)$ satisfies the triangle inequality.

Using the definition of $\oplus$ again, we get $b\in B$, $c\in C$ such that $b+c\geq s$. Because $s\geq \sup(B)$ and $s\geq \sup(C)$, in particular $s\geq b,c$, hence the triangle with distances $(b,c,s)$ satisfies the triangle inequality. Therefore, using the 4-values condition, there is $y\in S$ such that the triangles $(a,b,y)$ and $(c,z,y)$ satisfy the triangle inequality.

In particular, $a+b \geq y$ and $y+c\geq z$. Using Observation~\ref{obs:workswell}, we derive that $\hat a\oplus \hat b \supseteq \hat y$ and that $z\in \hat y\oplus \hat c$. Using monotonicity of $\oplus$, we get that $z\in (\hat a\oplus \hat b)\oplus \hat c$, contradicting the choice of $z$.
\end{proof}

Let $M=\{0\}\cup\mathfrak S$, where $0$ is a new element. Extend $\oplus$ by putting $0\oplus X=X$ and extend inclusion to a linear order $\mleq$ on $M$ by declaring $0\mleq X$ for every $X\in M$. By the previous proposition and the monotonicity observed above, $\Monoid=(M,\oplus,\mleq,0)$ is a distance monoid. From now on, we fix this monoid. We identify $\mathfrak S$ with the set of its positive elements and continue to use inclusion for their order. Note that Observation~\ref{obs:workswell} implies that if in an $\Monoid$-metric space there is a triangle with positive distances $X,Y,Z$, then $\sup(X)+\sup(Y) \geq \sup(Z)$. In particular, if $X=\hat x$, $Y=\hat y$ and $Z=\hat z$, then $x+y\geq z$.

\medskip

Later in the proof, we will need the following facts about $\Monoid$.

\begin{observation}\label{obs:looseness}
Let $A,B,C\in \mathfrak S$ such that $A$ is virtual and $C=\hat c$ is original.
\begin{enumerate}
\item If $A\oplus B\supseteq C$, then there is a number $a\in A$ such that $\hat a\oplus B\supseteq C$.
\item If $A\oplus A\supseteq C$, then there is a number $a\in A$ such that $\hat a\oplus \hat a\supseteq C$.
\end{enumerate}
\end{observation}
\begin{proof}
In the first case, by the definition of $\oplus$ and by the fact that $C$ is original, there are $a\in A$ and $b\in B$ such that $a+b\geq c$. The first statement follows directly from monotonicity. The second follows by applying the first with $A=B$ and choosing the larger of $a$ and $b$.
\end{proof}

\begin{observation}\label{obs:equivalences}
Let $A\in \mathfrak S$ be virtual. There is $a\in A$ such that $\hat a\oplus \hat a \supseteq A$. (Hence in particular $\hat a$ and $A$ are in the same block of $\Monoid$.)
\end{observation}
\begin{proof}
Put $a' = \sup(A)$. This means that for every $\varepsilon > 0$, there is $x\in A$ such that $a'-x<\varepsilon$. Hence, in particular, there is $a\in A$ such that $a'-a < \frac{a'}{2}$. Thus $a + a \geq a'$, and hence $\hat a\oplus \hat a\supseteq A$.
\end{proof}

\begin{lemma}\label{lem:orders}
Let $a,b\in S$. Then $a\sim b$ if and only if $\hat a\approx \hat b$.
\end{lemma}
\begin{proof}
First suppose that $a\sim b$ and without loss of generality assume that $a\leq b$. It is enough to find $n$ such that $n\times \hat a\supseteq \hat b$. From $a\sim b$, we get a sequence $x_1,x_2,\ldots,x_n$ such that $x_1 = a$, $x_n\geq b$, $x_i\in S$ and $x_{i}-x_{i-1} \leq a$ for every $i$. Also $\hat x_{i-1}\oplus \hat a = \{s\in S : s\leq x_{i-1}+a\}$ and $x_i \in \hat x_{i-1}\oplus \hat a$ because $x_i\leq x_{i-1}+a$. Thus $\hat x_i\subseteq \hat x_{i-1}\oplus \hat a$. Using induction, we then get that $n\times \hat a\supseteq \hat b$.

\medskip

Now suppose that $\hat a\approx \hat b$ and, without loss of generality, $a\leq b$. Choose $n$ such that $n\times\hat a\supseteq\hat b$. We prove by induction on $n$ that if $s\in n\times\hat a$ and $s\geq a$, then $a\sim s$. The case $n=1$ is clear. Suppose that $s\in (n+1)\times\hat a$ and $s\geq a$. There are $x\in n\times\hat a$ and $y\in\hat a$ such that $s\leq x+y$. If $x<a$, then $s-a<a$, and hence $a\sim s$. We can thus assume that $x\geq a$. By the induction hypothesis, $a\sim x$. If $s\leq x$, then $a\sim s$ because blocks are intervals. Otherwise $s-x\leq y\leq a$, so we can append $s$ to a sequence witnessing $a\sim x$. It follows that $a\sim s$. Since $b\in\hat b\subseteq n\times\hat a$, this gives $a\sim b$.
\end{proof}

\medskip

Let $\str A$ be an (ordered) $S$-metric space. We can define an (ordered) $\Monoid$-metric space $T(\str A)$ as follows:
\begin{enumerate}
\item The vertices of $T(\str A)$ are the same as the vertices of $\str A$,
\item the order is also the same (if there is one),
\item $d_{T(\str A)}(u,u)=0$ and $d_{T(\str A)}(u,v) = \widehat{d_\str A(u,v)}$ for $u\neq v$.
\end{enumerate}
The fact that $T(\str A)$ is an $\Monoid$-metric space follows from Observation~\ref{obs:workswell}. Note that if $\str A$ and $\str B$ are $S$-metric spaces and $f\colon A\to B$ is a function, then $f$ is an embedding $\str A\to \str B$ if and only if it is an embedding $T(\str A)\to T(\str B)$. We will now observe that orders convex with respect to block equivalences correspond to convex orders in the sense of Definition~\ref{defn:convex_monoid}:
\begin{observation}\label{obs:orders}
Let $\str A$ be an $S$-metric space and let $\leq_\str A$ be a linear order of vertices of $\str A$. The order $\leq_\str A$ is convex with respect to block equivalences if and only if it is convex in the sense of Definition~\ref{defn:convex_monoid} for $T(\str A)$.
\end{observation}
\begin{proof}
All positive distances in $T(\str A)$ are original.  By Lemma~\ref{lem:orders}, the function $\hat{ }$ induces a bijection between the equivalence classes of $\sim$ and the equivalence classes of $\approx$ on original distances. The observation follows by comparing the two definitions of convexity.
\end{proof}

Our goal is to use Theorem~\ref{thm:monoids} to obtain a Ramsey witness over $\Monoid$ and then adjust its distances back to $S$. We will use the following lemma.

\begin{lemma}\label{lem:finite_palette}
Let $\mathcal P\subseteq\mathfrak S$ be finite. There is a map $f\colon\mathcal P\to\widehat S$ such that:
\begin{enumerate}
\item\label{palette:inside} $f(X)\subseteq X$ for every $X\in\mathcal P$,
\item\label{palette:original} $f(X)=X$ for every original $X\in\mathcal P$,
\item\label{palette:triangles} if $X,Y,Z\in\mathcal P$ form a metric triangle, then $f(X),f(Y),f(Z)$ form a metric triangle, and
\item\label{palette:blocks} $f(X)\approx X$ for every $X\in\mathcal P$.
\end{enumerate}
\end{lemma}
\begin{proof}
We proceed by induction on the number of virtual elements of $\mathcal P$. If there are none, take $f$ to be the identity.

Suppose that $\mathcal P$ contains a virtual element and let $X$ be its largest virtual element (with respect to inclusion). Put
$$
\mathcal X_0=\{Y\in\mathcal P:Y\subsetneq X\}.
$$
Further, let $\mathcal X_1$ consist of all pairs $(Y,Z)\in\mathcal P^2$ such that $Y\neq X$, $Z\supsetneq X$, $Z\supsetneq Y$, and $X,Y,Z$ form a metric triangle, and let $\mathcal X_2$ consist of all $Z\in\mathcal P$ such that $Z\supsetneq X$ and $X,X,Z$ form a metric triangle. Since $X$ is the largest virtual element of $\mathcal P$, every $Z$ occurring in $\mathcal X_1$ or $\mathcal X_2$ is original.

Use Observation~\ref{obs:looseness} to choose $a_{Y,Z}\in X$ for every $(Y,Z)\in\mathcal X_1$ such that $\hat a_{Y,Z}\oplus Y\supseteq Z$, and also $a_Z\in X$ for every $Z\in\mathcal X_2$ such that $\hat a_Z\oplus\hat a_Z\supseteq Z$. Finally, let $a_X\in X$ be given by Observation~\ref{obs:equivalences}. Choose $\alpha\in X$ sufficiently large so that
\begin{enumerate}
\item[(i)] $\alpha\geq\sup(Y)$ for every $Y\in\mathcal X_0$,
\item[(ii)] $\alpha\geq a_{Y,Z}$ for every $(Y,Z)\in\mathcal X_1$,
\item[(iii)] $\alpha\geq a_Z$ for every $Z\in\mathcal X_2$, and
\item[(iv)] $\alpha\geq a_X$.
\end{enumerate}
Such an $\alpha$ exists. Indeed, all the sets above are finite; all $a_{Y,Z}$, $a_Z$, and $a_X$ belong to $X$; and if $Y\subsetneq X$, then $\sup(Y)<\sup(X)$ because $X$ is virtual.

Define $h\colon\mathcal P\to\mathfrak S$ by $h(X)=\hat\alpha$ and $h(Y)=Y$ for $Y\neq X$. Clearly $h(Y)\subseteq Y$ for every $Y\in\mathcal P$, it fixes all original elements of $\mathcal P$, and by (iv) we have $h(Y)\approx Y$ for every $Y\in\mathcal P$.

We claim that $h$ also preserves metric triangles. Since $h(X)\subseteq X$, only a triangle inequality in which an occurrence of $X$ is on the right-hand side can become false. Let $D$ be the left-hand side of such an inequality and let $Y$ be the other distance on the right-hand side. If $D\subsetneq X$, then $D\in\mathcal X_0$, and (i) gives $D\subseteq\hat\alpha$. If $D=X$, the inequality is automatic by monotonicity. We may thus assume that $D\supsetneq X$. Since $X$ is the largest virtual element, $D$ is original. If $Y\supseteq D$, the inequality is automatic. Otherwise $D\supsetneq Y$. If $Y=X$, then $D\in\mathcal X_2$ and the inequality follows from (iii); otherwise $(Y,D)\in\mathcal X_1$ and it follows from (ii). This proves the claim.

Put $\mathcal P'=h[\mathcal P]$. It contains one fewer virtual element than $\mathcal P$, so by the induction hypothesis there is a map $g\colon\mathcal P'\to\widehat S$ satisfying properties~\ref{palette:inside}--\ref{palette:blocks}. Put $f=g\circ h$. Since $h(X)=\hat\alpha$ is original, it is fixed by $g$. Properties~\ref{palette:inside} and~\ref{palette:original} follow immediately. Property~\ref{palette:triangles} follows because both $h$ and $g$ preserve metric triangles, and property~\ref{palette:blocks} follows by transitivity of $\approx$.
\end{proof}

The following proposition is now an immediate consequence.
\begin{prop}\label{prop:main}
Let $\str A$ be a finite $\Monoid$-metric space. There is an $S$-metric space $U(\str A)$ satisfying the following:
\begin{enumerate}
\item $\str A$ and $U(\str A)$ have the same set of vertices.
\item\label{main:2} $\Aut(\str A) \subseteq \Aut(U(\str A))$.
\item\label{main:3} For all distinct $u,v\in A$, $d_{U(\str A)}(u,v)\in d_\str A(u,v)$. If $d_\str A(u,v) = \hat s$ for some $s\in S$, that is, if $d_\str A(u,v)$ is original, then $d_{U(\str A)}(u,v) = s$.
\item\label{main:4} A linear order $\leq$ on $A$ is convex with respect to block equivalences of $U(\str A)$ if and only if it is convex for $\str A$ in the sense of Definition~\ref{defn:convex_monoid}.
\end{enumerate}
\end{prop}
Note that the conditions imply that if $\str C$ is an $S$-metric space, then $U(T(\str C)) = \str C$. Also, if all positive distances of an $\Monoid$-metric space $\str A$ are original, then $T(U(\str A))=\str A$.
\begin{proof}
Apply Lemma~\ref{lem:finite_palette} to $\mathcal P=\operatorname{rng}(d_\str A)\setminus\{0\}$. For every $X\in\mathcal P$, write $f(X)=\hat s_X$ and put $d_{U(\str A)}(u,u)=0$ and, for $u\neq v$,
$$
d_{U(\str A)}(u,v)=s_{d_\str A(u,v)}.
$$
Lemma~\ref{lem:finite_palette}(\ref{palette:triangles}) says exactly that this is an $S$-metric space. Since the new distance depends only on the old distance, every automorphism of $\str A$ is an automorphism of $U(\str A)$, which proves~(\ref{main:2}). Lemma~\ref{lem:finite_palette}(\ref{palette:inside}) and~Lemma~\ref{lem:finite_palette}(\ref{palette:original}) give~(\ref{main:3}). Finally, by Lemma~\ref{lem:finite_palette}(\ref{palette:blocks}), every positive distance is replaced by one in the same block of $\Monoid$. Hence $\str A$ and $T(U(\str A))$ induce the same block equivalences on vertices, and~(\ref{main:4}) follows from Observation~\ref{obs:orders}.
\end{proof}
Note that it is even possible to ensure that $\sup(d_\str A(u,v))-d_{U(\str A)}(u,v)<\varepsilon$ for any fixed $\varepsilon > 0$ and all distinct $u,v\in A$.

Now, we are ready to prove Theorem~\ref{thm:main}.
\begin{proof}[Proof of Theorem~\ref{thm:main}]
Let $\str A$, $\str B$ be finite ordered $S$-metric spaces whose orders are convex with respect to block equivalences. By Observation~\ref{obs:orders}, $T(\str A)$ and $T(\str B)$ are convexly ordered $\Monoid$-metric spaces. Hence, by Theorem~\ref{thm:monoids}, there is a convexly ordered $\Monoid$-metric space $\str C'$ such that
$$\str C'\longrightarrow (T(\str B))^{T(\str A)}_2.$$
Let $\str C = U(\str C')$ be given by Proposition~\ref{prop:main}, with the same order. By Proposition~\ref{prop:main}(\ref{main:4}), $\str C$ is an $S$-metric space convexly ordered with respect to block equivalences.

Suppose that $c\colon \Emb(\str A,\str C)\to 2$ is a colouring. Every embedding $f\colon T(\str A)\to\str C'$ is also an embedding $f\colon\str A\to\str C$: all positive distances of $T(\str A)$ are original, and Proposition~\ref{prop:main}(\ref{main:3}) turns every original distance $\hat s$ in $\str C'$ into the distance $s$ in $\str C$. Thus $c$ restricts to a colouring of $\Emb(T(\str A),\str C')$. By the choice of $\str C'$, there is an embedding $e\colon T(\str B)\to\str C'$ such that this restricted colouring is constant on $e\circ\Emb(T(\str A),T(\str B))$. All positive distances in the image of $e$ are original, so the same map $e$ is also an embedding $\str B\to\str C$. Moreover, every embedding $g\colon\str A\to\str B$ is also an embedding $T(\str A)\to T(\str B)$. Hence $c$ is constant on $e\circ\Emb(\str A,\str B)$. This proves the Ramsey property. The expansion property follows by the standard argument (see e.g.~\cite[Theorem~5.17]{Konecny2018b}).
\end{proof}

\bibliographystyle{alpha}
\bibliography{ramsey}
\end{document}